\documentclass{amsart}
\usepackage{amsmath}
\usepackage{amssymb}
\usepackage{amsthm}
\usepackage{mathrsfs}
\newtheorem{thm}{Theorem}[section]
\newtheorem{lem}[thm]{Lemma}

\usepackage{mathtools}
\mathtoolsset{showonlyrefs=true}

\date{today}
\begin{document}
\title[Homogeneous Boltzmann equation with fractional Laplacian term]
{Some remarks on the homogeneous Boltzmann equation with the fractional Laplacian term
}
\author{Shota Sakamoto}
\address{Graduate School of Human and Environmental Studies,
Kyoto University,
Kyoto, 606-8501, Japan} \email{sakamoto.shota.76r@st.kyoto-u.ac.jp}



\date{}

\begin{abstract}
We study the homogeneous Boltzmann equation with the fractional Laplacian term. Working on the Fourier side we solve the resulting integral equation, and improve a previous result by Y.-K. Cho. We replace the initial data space with a certain space $\mathcal{M}^\alpha$ introduced by Morimoto, Wang, and Yang. This space precisely captures the Fourier image of probability measures with bounded fractional moments, providing a more natural initial condition. We show existence of a unique global solution, in addition to the expected maximal growth estimates and  stability estimates. As a consequence we obtain a continuous density solution of the original equation.
\end{abstract}

\maketitle

\section{Introduction}

We consider the spatially homogeneous Boltzmann equation with a diffusion term, that is,

\begin{align}\label{BE}
\begin{cases}
\partial_t f(v,t) +\delta_p (-\Delta)^{{p}/{2}}f(v,t)=Q(f,f)(v,t), \ (v,t) \in \mathbb{R}^3 \times (0,\infty ),\\
f(v,0)=f_0(v),
\end{cases}
\end{align}
where $0< p \le 2$ and $\delta_p \ge 0$ are constants. This equation is a variant of the homogeneous Boltzmann equation
\begin{align}\label{eq:original}
\begin{cases}
\partial_t f(v,t)=Q(f,f)(v,t), \ (v,t) \in \mathbb{R}^3 \times (0,\infty ),\\
f(v,0)=f_0(v).
\end{cases}
\end{align}
Equation \eqref{BE} was recently studied by Cho~\cite{C}. He discussed the Fourier transform of \eqref{BE} and showed that if we take an initial datum from a certain space (which we shall define later), then there exists a unique global solution on the Fourier side and this solution satisfies a maximal stability estimate. The estimate is a consequence of the diffusion term, so such an estimate is likely not valid for solutions of the Fourier transform of \eqref{eq:original}.
It should be noted that the initial data space used by Cho~\cite{C} may contain "undesirable" data for this problem, 
that is, data which prevents us from inverting the Fourier transformed problem back into the original one (see, for example, Remark 1(ii)
of \cite{C}).   In order to avoid this obstruction, in this paper we 
instead use a more suitable space introduced by Morimoto-Wang-Yang~\cite{MWY}, as an initial data space.  Using that space, we obtain the existence of a time-continuous unique global solution of \eqref{BE}. Readers may refer to \cite{C,CK,CK2,M,MWY} and the references therein for historical progress and recent research on these spaces.

We now discuss \eqref{BE} in greater detail. $f(v,t)$ is the density distribution of particles in rarefied gas with velocity $v \in \mathbb{R}^3$ and time $t>0$.
The right-hand side of \eqref{BE} is defined by the Boltzmann collision operator corresponding to the 
the Maxwellian molecule type cross section, 
\begin{align}\label{CO}
Q(f,g)(v)= \int_{\mathbb{R}^3} \int_{\mathbb{S}^2}b(\mathbf{k}\cdot\sigma)(f'_*g'-f_*g)d\sigma dv_*,
\end{align}
where $f'_*=f(v_*'),\ g'=g(v'),\ f_*=f(v_*),  g=g(v),\,\mathbf{k}=(v-v_*)/(\vert v-v_* \vert)$, and
\begin{align*}
v'=\frac{v+v_*}{2}+\frac{\vert v-v_*\vert}{2}\sigma,\ v_*'=\frac{v+v_*}{2}-\frac{\vert v-v_* \vert}{2}\sigma.
\end{align*}
The Maxwellian molecule type cross section $b(\mathbf{k}\cdot \sigma)$ in \eqref{CO} is a non-negative function depending only on 
the deviation angle $\theta = \cos^{-1}(\mathbf{k}\cdot \sigma)$. As in Villani~\cite{V}, $\theta$ is customarily restricted to the range $[0,\pi /2]$
by  replacing $b(\mathbf{k}\cdot \sigma) = b(\cos \theta)$ with its ``symmetrised version''
\begin{align*}
\Big (b(\cos \theta) + b(\cos(\pi-\theta))\Big){\bf {1}}_{\{0 \le \theta \le \pi/2\}}.
\end{align*}
Since the physical model $b(\cos \theta)$ coming from the inverse power laws has a singularity near $\theta=0$, that is,
\begin{align}\label{Sing}
b(\cos \theta ) \sim K \theta^{-2-2s} \quad (0<s<1, \ K>0),
\end{align}
we shall throughout this note assume the following weak integrability condition 
\begin{align}\label{Weak integrability}
\exists \alpha_0 \in (0,2] \text{ such that }
\int^{{\pi}/{2}}_0 b(\cos \theta)\sin \theta \sin^{\alpha_0} \left(\frac{\theta}{2}\right) d\theta < \infty.
\end{align}
We note that
\eqref{Weak integrability} is satisfied for any cross section $b$ with property \eqref{Sing} as long as $\alpha_0 >2s$,
although $b(\cos\theta)\sin \theta$ is not integrable on $[0,\pi/2]$.
These constraints coming from the physical model will be used throughout the paper.

Next, we reformulate equation \eqref{BE}. By taking into account the Bobylev formula, we apply the Fourier transform to \eqref{BE} and obtain
\begin{align}\label{eq:FTint-difeq}
\begin{cases}
(\partial_t+ \delta_p \vert \xi \vert^p)\phi(\xi,t)= \mathcal{B}(\phi)(\xi,t), \ (\xi ,t) \in \mathbb{R}^3 \times (0,\infty),\\
\phi(\xi,0)=\phi_0(\xi),
\end{cases}
\end{align}
where $\phi(\xi,t)=\mathcal{F}[f(\cdot, t)](\xi)$,
\begin{align}
\mathcal{B}(\phi)(\xi)=\mathcal{F}[Q(f,f)](\xi)=\int_{\mathbb{S}^2} b\left(\frac{\xi}{\vert \xi \vert}\cdot \sigma\right) [\phi(\xi^+)\phi(\xi^-)-\phi(0)\phi(\xi)] d\sigma,
\end{align}
and $\xi^{\pm}=(\xi\pm \vert \xi \vert\sigma)/2$.
(Thus, in \eqref{eq:FTint-difeq} we have $\mathcal{B}(\phi)(\xi,t)=\mathcal{B}(\phi(\cdot, t))(\xi)$ and we will use similar abuse of notation in the sequel without further comment.)
The corresponding integral equation is
\begin{align}\label{FE}
\phi(\xi,t)=e^{-\delta_p\vert\xi\vert^pt} \phi_0(\xi) +\int^t_0 e^{-\delta_p\vert\xi\vert^p(t-\tau)} \mathcal{B}(\phi)(\xi,\tau) d\xi.
\end{align}
We assume that $f(\cdot,t)$ is a probability measure on $\mathbb{R}^3$ for every $t>0$ at least, thus $\phi(\xi,t)$ is equal to $\int_{\mathbb{R}^3_v} e^{-iv\cdot \xi} df(v,t)$ and $\phi(0,t)=1$.

Before we discuss known results concerning \eqref{eq:FTint-difeq}, 
concider first the Fourier transform of \eqref{eq:original}, namely,
\begin{align}\label{simplifiedhomogeq}
\begin{cases}
\partial_t\phi(\xi,t)= \mathcal{B}(\phi)(\xi,t), \ (\xi ,t) \in \mathbb{R}^3 \times (0,\infty),\\
\phi(\xi,0)=\phi_0(\xi).
\end{cases}
\end{align}
Define $P_\alpha(\mathbb{R}^d)\ (d\ge 1,\ \alpha \ge 0)$ as the set of all probability measures on $\mathbb{R}^d$ with finite moment of order $\alpha$. In other words, if $\mu \in P_\alpha(\mathbb{R}^d)$ then $\mu$ satisfies
\begin{align*}
\int_{\mathbb{R}^d} d\mu(v) =1,\quad \int_{\mathbb{R}^d} \vert v \vert^\alpha d\mu(v) <\infty.
\end{align*}
When $\alpha \ge 1$, we add on the condition
\begin{align*}
\int_{\mathbb{R}^d} v_jd\mu(v)=0\quad (j=1,2,\cdots,d)
\end{align*}
to the definition of $P_\alpha(\mathbb{R}^d)$.
Since we expect a solution of \eqref{eq:original} to conserve its mass, momentum, and energy, we thus normalize so that solutions are probability measures with zero mean and finite moment of order $2$ (or, finite variance). 
In order to establish a solution with infinite energy, Cannone and Karch \cite{CK} invented spaces of characteristic functions, namely
\begin{align*}
\mathcal{K} &=\left\{ \phi : \mathbb{R}^d \rightarrow \mathbb{C} \ \left\vert \ \phi(\xi)= \int_{\mathbb{R}^d} e^{-iv\cdot \xi}df(v),\ f \in P_0(\mathbb{R}^d) \right.\right\},\\
\mathcal{K}^\alpha &=\left\{ \phi \in \mathcal{K} \ \left\vert \ \Vert \phi-1 \Vert_{\alpha} = \sup_{\xi \in \mathbb{R}^d} \frac{\vert \phi(\xi)-1\vert}{\vert\xi\vert^\alpha} < \infty \right.\right\},
\end{align*}
and analysed the homogeneouse Boltzmann equation with the Maxwellian molecule type collision kernel. $\mathcal{K}^\alpha$ is a complete metric space endowed with the norm
\begin{align*}
\Vert \phi -\psi \Vert_{\alpha}&=\sup_{\xi \in \mathbb{R}^3} \frac{\vert \phi(\xi)-\psi(\xi) \vert}{\vert\xi\vert^{\alpha}} \quad (\phi,\ \psi \in \mathcal{K}^\alpha).
\end{align*}
Since 
\begin{equation}
\{ 1 \}= \mathcal{K}^2 \subset \mathcal{K}^{\alpha_1} \subset \mathcal{K}^{\alpha_2} \subset \mathcal{K}^0=\mathcal{K}\quad(0 \le \alpha_2 \le \alpha_1 \le 2),
\end{equation}
the $\mathcal{K}^\alpha$-valued solutions they established have infinite energy. This result was modified by Morimoto~\cite{M} by using finer calculations, and by Morimoto-Wang-Yang~\cite{MWY} by introducing the space
\begin{align*}
\mathcal{M}^\alpha =\left\{ \phi \in \mathcal{K}\ \left\vert \Vert \phi-1\Vert_{\mathcal{M}^\alpha}=\int_{\mathbb{R}^d} \frac{\vert \phi(\xi)-1\vert}{\vert \xi \vert^{d+\alpha}} d\xi < \infty \right.\right\}.
\end{align*}
When $0<\beta \le \alpha < 2$, this is a complete metric space endowed with
\begin{align*}
\mathrm{dis}_{\alpha,\beta}(\phi,\psi) & =\Vert \phi-\psi\Vert_{\mathcal{M}^\alpha} +\Vert\phi-\psi \Vert_\beta,\\
 \Vert \phi -\psi \Vert_{\mathcal{M}^\alpha} & =\int_{\mathbb{R}^3} \frac{\vert \phi(\xi)-\psi(\xi) \vert}{\vert\xi\vert^{d+\alpha}}d\xi \quad(\phi,\ \psi \in \mathcal{M}^\alpha).
\end{align*}
For any $\alpha \in [0,2]$, $\mathcal{M}^\alpha\subset \mathcal{F}(P_\alpha(\mathbb{R}^d))( \subsetneq \mathcal{K}^\alpha)$. Moreover, if $\alpha \neq 1$, then $\mathcal{M}^\alpha = \mathcal{F}(P_\alpha(\mathbb{R}^d))$. This is the reason why $\mathcal{M}^\alpha$ is preferable to $\mathcal{K}^\alpha$ for our problem, because this space has a simple interpretation: when we consider the original Boltzmann equation we can take initial data from $P_\alpha(\mathbb{R}^d)$ and analyse the equation on the Fourier side. The result is Theorem 1.4 in \cite{MWY}. This theorem gives the existence of a unique global measure valued solution of \eqref{eq:original} in $C([0,\infty); P_\alpha(\mathbb{R}^3))$ with a $P_\alpha(\mathbb{R}^3)$-valued initial datum.

Now let us recall some known results concerning \eqref{eq:FTint-difeq}. 
Using techniques from \cite{CK} and \cite{M} (in which \eqref{simplifiedhomogeq} is studied)
Cho \cite{C} proved that \eqref{eq:FTint-difeq} has a solution in $\mathcal{S}^{\alpha}(\mathbb{R}^3 \times [0,\infty))$. A stability estimate and a maximum growth estimate were also shown. Here, for all $T>0$,
\begin{align*}
&\mathcal{S}^{\alpha}(\mathbb{R}^3 \times [0,T])  =\left\{ \phi \in C([0,T];\mathcal{K}^\alpha)  \ \left\vert \  \phi(\xi,\cdot) \in C([0,T])\right.\right., \\
& \qquad \qquad \qquad \qquad \qquad \qquad \qquad \qquad \quad \left.\partial_t \phi(\xi,\cdot) \in C((0,T))\ \mathrm{for\ }\ \forall\xi \in \mathbb{R}^3 \right\},\\
&\mathcal{S}^\alpha(\mathbb{R}^3 \times [0,\infty))  = \bigcup_{T>0} \mathcal{S}^{\alpha}(\mathbb{R}^3 \times [0,T]).
\end{align*}

We shall consider \eqref{FE} with initial data $\phi_0$ in $\mathcal{M}^\alpha$
$(\alpha \in [0,2])$.
We apply the technique developed in \cite{MWY} to discuss \eqref{FE}.
Our main theorems are the following.
\begin{thm}\label{existence}
Let $\alpha_0 \le \alpha <p$. Then for any $\phi_0 \in \mathcal{M}^{\alpha}$, \eqref{FE} has a classical solution $\phi$ in
\begin{align*}
& \mathcal{T}^{\alpha}(\mathbb{R}^3\times [0,\infty))=\left\{ \phi \in C ([0,\infty);\mathcal{M}^{\alpha}) \left\vert \phi(\xi,\cdot)\in C([0,\infty)),\right.\right.\\
& \qquad \qquad \qquad \qquad \qquad \qquad \qquad \qquad \qquad \left.\partial_t \phi(\xi,\cdot)\in C((0,\infty)) \ \text{for }\forall \xi \in \mathbb{R}^3 \right\}.
\end{align*}
This solution $\phi$ satisfies the a priori estimate
\begin{align}\label{maximal growth}
\sup_{\xi \in \mathbb{R}^3} e^{\delta_p\vert\xi\vert^pt} \vert \phi(\xi,t)\vert \le 1 \quad \text{for each } t \ge 0.
\end{align}
\end{thm}
Since $\phi$ in Theorem \ref{existence} satisties $\partial_t \phi(\xi,\cdot)\in C((0,\infty))$, $\phi$ is also a solution of \eqref{eq:FTint-difeq}. \eqref{maximal growth} is a consequence of adding the diffusion term $\delta_p (-\Delta)^{{p}/{2}}$. When $\delta_p=0$, \eqref{maximal growth} just gives the obvious statement $\vert \phi(\xi,t) \vert \le 1$ for any $\xi$ and $t$. This makes difference between \eqref{BE} and \eqref{eq:original}.

In order to state our stability result, we introduce the notation $\mathcal{T}^\alpha_p(\mathbb{R}^3\times [0,\infty))$ for all elements of $\mathcal{T}^\alpha(\mathbb{R}^3\times [0,\infty))$ which satisfy \eqref{maximal growth}.
The following constants will also be needed in the sequel. For $2s < \alpha_0 \le \alpha <p\le 2$, we define
\begin{align*}
\gamma_{\alpha}&=2\pi \int^{\pi/2}_0 b(\cos \theta)\sin \theta \left[\cos^{\alpha}\left(\frac{\theta}{2}\right)+\sin^{\alpha} \left(\frac{\theta}{2}\right) \right] d\theta,\\
\lambda_{\alpha}&=2\pi \int^{\pi/2}_0 b(\cos \theta)\sin \theta \left[\cos^{\alpha} \left(\frac{\theta}{2}\right)+\sin^{\alpha} \left(\frac{\theta}{2}\right)-1\right] d\theta,\\
\mu_{\alpha}&=2\pi \int^{\pi/2}_0 b(\cos \theta)\sin \theta \sin^{\alpha} \left(\frac{\theta}{2}\right) d\theta,\\
C_{p,\alpha}&= \int_{\mathbb{R}^3} \frac{1-e^{-\vert \xi \vert^p}}{\vert \xi \vert^{3+\alpha}} d\xi= \frac{4\pi}{\alpha}\Gamma \left(1-\frac{\alpha}{p}\right),
\end{align*}
where $\Gamma$ is the gamma function. $\gamma_\alpha$ is used to discuss \eqref{FE} under the cutoff assumption, that is, 
when $b\in L^1(\mathbb{S}^2)$. Easily we see that $\gamma_2\le \gamma_{\alpha}\le 2\gamma_2$. When we discuss the equation under the non-cutoff assumtion, we use $\lambda_\alpha$ and $\mu_\alpha$. In this case, $0 \le \lambda_\alpha < \infty$ (see \cite{CK}) and $0 \le \mu_\alpha < \infty$ (due to \eqref{Weak integrability}). The second identity for $C_{p,\alpha}$, which was mentioned to us by Cho, shows that $C_{p,\alpha}$ is finite if and only if $0<\alpha<p$.
The identity is easy to prove using a change of variables followed by integration by parts.

\begin{thm}\label{stability}
Under the assumptions of Theorem \ref{existence}, let $\phi,\psi \in \mathcal{T}^{\alpha}_p(\mathbb{R}^3\times [0,\infty))$ be solutions of \eqref{FE} with initial data $\phi_0,\psi_0 \in \mathcal{M}^{\alpha}$ respectively. Then 
the following stability estimate holds:
\begin{align*}
\left\Vert e^{\delta_p \vert \cdot \vert^pt}(\phi(\cdot,t)-\psi(\cdot,t))\right\Vert_{\mathcal{M}^{\alpha}} \le e^{\lambda_{\alpha}t} \Vert \phi_0-\psi_0\Vert_{\mathcal{M}^{\alpha}}.
\end{align*}
In particular, \eqref{FE} has at most one solution in $\mathcal{T}^\alpha_p(\mathbb{R}^3\times [0,\infty))$.
\end{thm}

Interpreting Theorem \ref{stability} in terms of the original equation yields the following result.

\begin{thm}\label{thm:original}
Let $\alpha_0 \le \alpha < p$, $\alpha \neq 1$, and $f_0 \in P_\alpha(\mathbb{R}^3)$. Let $\phi$ be the unique global solution of \eqref{eq:FTint-difeq} with initial datum $\phi_0 =\mathcal{F}[f_0] \in \mathcal{M}^\alpha$ given by Theorems \ref{existence} and \ref{stability}. Then $f(v,t)=\mathcal{F}^{-1}[\phi(\cdot,t)](\xi)$ is a unique global solution of \eqref{BE} with initial datum $f_0$. This inverse Fourier transform is well-defined. $f(v, t)$ is a smooth probability density of $v$ and tends to $0$ as $\vert v \vert\to\infty$ for each $t>0$. Moreover, $f \in C([0,\infty); P_\alpha(\mathbb{R}^3))$, where time continuity of $f$ is interpreted in the following sense: for any $t_0>0$ and $\psi \in C(\mathbb{R}^3)$ satisfying $ \vert \psi(v) \vert \le C(1+ \vert v \vert^2)^{\alpha/2}$ for some positive $C$,
\begin{equation}
\lim_{t\rightarrow t_0} \int_{\mathbb{R}^3} \psi(v) f(v,t)dv = \int_{\mathbb{R}^3} \psi(v) f(v,t_0)dv.
\end{equation}
\end{thm}
Note that a slightly stronger assumption is needed to estabilsh a solution with the same properties when $\alpha =1$ (see \cite[Remark 1.5]{MWY}).

The outline of this paper is as follows. We cite some lemmas in Section 2. We modify some of them so that they are applicable to our problem. In Section 3 we prove that under the cutoff assumption we have a global unique solution of \eqref{FE}, and deduce an a priori estimate and a stability estimate. These results are used to discuss \eqref{FE} under the non-cutoff assumption in Section 4. The proof of main theorems are shown in this section. In Section 5 we discuss non-existence of solutions in the case $p \le \alpha \le 2$.

\section{Preliminaries}

Under the cutoff assumption, $Q$ is split into a gain term $Q^+$ and a loss term $Q^-$ in a self-evident way. We denote the Fourier transform of the gain term $Q^+$ by
\begin{align*}
\mathcal{G}(\phi)(\xi)=\int_{\mathbb{S}^2}b\left(\frac{\xi}{\vert\xi\vert}\cdot \sigma \right)\phi(\xi^+)\phi(\xi^-) d\sigma.
\end{align*}

First, we cite two lemmas from Morimoto-Wang-Yang~\cite{MWY}.
\begin{lem}[Morimoto-Wang-Yang]\label{MWY-1}
Assume $b \in L^1(\mathbb{S}^2)\ ( \Leftrightarrow \gamma_2 < \infty)$. For all $\ \phi$ and $\psi \in \mathcal{M}^{\alpha}$,
\begin{align*}
\Vert \mathcal{G}(\phi)-\mathcal{G}(\psi) \Vert_{\mathcal{M}^{\alpha}} \le \gamma_{\alpha} \Vert \phi - \psi \Vert_{\mathcal{M}^{\alpha}}.
\end{align*}
\end{lem}

\begin{lem}[Morimoto-Wang-Yang]\label{MWY-2}
There exists a constant $C_0>0$ independent of $\phi \in \mathcal{M}^{\alpha}$, such that
\begin{align*}
\int_{\mathbb{R}^3} \frac{\vert \mathcal{B}(\phi)(\xi)\vert}{\vert \xi \vert^{3+\alpha}}d\xi \le C_0 \mu_{\alpha} \Vert \phi -1 \Vert_{\mathcal{M}^{\alpha}}.
\end{align*}
\end{lem}

Then, we show a lemma concerning $e^{-\vert\xi\vert^p}$, which is a classical result in probability theory.

\begin{lem}\label{stable process}
$e^{-\vert\xi\vert^pt}\ (\xi \in\mathbb{R}^3,\ t\ge 0)$ is a positive definite function of $\xi$ if and only if $0<p\le2$. For every $t>0$, $f_p(v,t)$, which is defined as 
\begin{align*}
f_p(v,t)=\frac{1}{2\pi^2}\int^{\infty}_0 e^{-r^pt}r^2 \frac{\sin r\vert v\vert}{r\vert v \vert}dr,
\end{align*}
is a probability density on $\mathbb{R}^3$, and $\mathcal{F}[f_p(\cdot,t)]=e^{-\vert\xi\vert^pt}$. Futhermore,
\begin{align*}
\int_{\mathbb{R}^3} \vert v \vert^{\alpha}f_p(v,t)dv <\infty \ (0\le \alpha <p) \text{ and } \int_{\mathbb{R}^3} \vert v \vert^{p}f_p(v,t)dv=\infty.
\end{align*}
\end{lem}
\begin{proof}
Positive definiteness is proved in Schoenberg~\cite{S}. For each $t>0$, $e^{-\vert\xi\vert^pt}$ is a rapidly decreasing radial function defined on $\mathbb{R}^3$. 
Its inverse Fourier transform is
\begin{equation}\label{eq:firststepFT}
\mathcal{F}^{-1}[e^{-\vert\cdot\vert^pt}](v)=\frac{1}{(2\pi)^3} (2\pi)^{\frac{3}{2}} \vert v \vert^{-\frac{1}{2}} \int^{\infty}_0 J_{\frac{1}{2}}(r\vert v \vert)e^{-r^pt} r^{\frac{3}{2}}dr,
\end{equation}
where $J_{\nu}(z)\ (\nu \in \mathbb{C},\ \mathrm{Re}(\nu)>-1/2)$ is the Bessel function
\begin{align*}
J_{\nu}(z)=\left[ \Gamma\left( \frac{1}{2} \right) \Gamma \left(\nu+\frac{1}{2}\right)\right]^{-1} \left(\frac{z}{2}\right)^{\nu} \int^1_{-1} (1-t^2)^{\nu-\frac{1}{2}}e^{izt} dt
\end{align*}
(see Taylor~\cite{T} for instance). A straigtforward calculation of $J_{\frac{1}{2}}(r\vert v\vert)$ together with \eqref{eq:firststepFT} gives
\begin{equation}
\mathcal{F}^{-1}[e^{-\vert\cdot\vert^pt}](v)
=\frac{1}{2\pi^2}\int^{\infty}_0 e^{-r^pt}r^2 \frac{\sin r\vert v\vert}{r\vert v \vert}dr
=f_p(v,t).
\end{equation}
Hence, $\mathcal{F}[f_p(\cdot,t)]=e^{-\vert\xi\vert^pt}$.

We now turn to the moment estimates. A simple calculation shows that
$f_p(v,t)=t^{-3/p}f_p(t^{-1/p}v,1)$, so it suffices to consider the case $t=1$.
From \cite[Theorem 2.1]{BG} we have
\begin{align*}
\lim_{\vert v\vert \rightarrow \infty}\vert v\vert^{3+p} f_p(v,1)=\frac{p2^{p-1}}{\pi^{\frac{5}{2}}} \sin \left( \frac{p \pi}{2}\right) \Gamma\left(\frac{3+p}{2}\right) \Gamma\left(\frac{p}{2}\right),
\end{align*}
and this leads to the desired estimates.
\end{proof}

Cho~\cite[Section 3]{C}
discusses conditions for which $\mathcal{G}(\phi)$ and $\mathcal{B}(\phi)$ make sense. Since $\mathcal{M}^{\alpha} \subset \mathcal{K}^{\alpha}\subset \mathcal K$ $(\alpha \ge 0)$,
it follows that Lemma 3.1 and Lemma 3.3 in \cite{C}
are still valid when we take $\phi$ from $\mathcal{M}^{\alpha}$.
We state these facts as Lemmas \ref{G} and \ref{B}.

\begin{lem}\label{G}
Assume $b \in L^1(\mathbb{S}^2)$ and $T>0$. Then  $\mathcal{G}(\phi)(\xi)$ is a continuous positive definite function of $\xi$ for each $\phi \in \mathcal{M}^\alpha$.

Moreover, if $\phi(\cdot, t) \in \mathcal{M}^\alpha$ for each $t \in [0,T]$ and $\phi(\xi,\cdot)\in C([0,T])$ for each $\xi \in \mathbb{R}^3$, then $\mathcal{G}(\phi)(\xi,t) \in C(\mathbb{R}^3 \times [0,T])$.
\end{lem}

\begin{lem}\label{B}
Assume $\mu_{\alpha_0} < \infty$ and $\ T>0$. If $\phi \in C([0,T];\mathcal{M}^{\alpha})$ and $ \phi(\xi,\cdot)\in C([0,T])$ for each $\xi \in \mathbb{R}^3$, then $\mathcal{B}(\phi)(\xi,\cdot)\in C([0,T])$ for each $\xi \in \mathbb{R}^3$.
\end{lem}

\section{Existence and behaviour of a solution under the cutoff assumption}

In order to consider the non-cutoff case in Section \ref{section:noncutoff}, we first prove 
that under the cutoff assumption ($\gamma_2 <\infty$), \eqref{FE} has a unique solution in $\mathcal{T}^{\alpha}(\mathbb{R}^3\times [0,\infty))$. We also provide a stability estimate and an a priori estimate for this solution.
\subsection{Existence}

Assume $b\in L^1(\mathbb{S}^2)$. Under the cutoff assumption, \eqref{FE} is written as
\begin{align}\label{FE cutoff}
\phi(\xi,t)=e^{-(\gamma_2+\delta_p \vert \xi \vert^p)t}\phi_0(\xi)+\int^t_0 e^{-(\gamma_2+\delta_p\vert \xi \vert^p)(t-\tau)} \mathcal{G}(\phi)(\xi,\tau)d\tau.
\end{align}
In this subsection we will prove the following theorem.
\begin{thm}\label{thm:Existence}
Let $\alpha_0 \le \alpha <p,\  b\in L^1(\mathbb{S}^2)$ and $\phi_0 \in \mathcal{M}^{\alpha}$. Then  \eqref{FE cutoff} has a solution $\phi \in \mathcal{T}^{\alpha}(\mathbb{R}^3\times [0,\infty))$, which satisfies
\begin{align*}
\sup_{\xi \in \mathbb{R}^3} e^{\delta_p\vert\xi\vert^pt} \vert \phi(\xi,t)\vert \le 1 \text{ for each}\ t \ge 0.
\end{align*}
\end{thm}

For a fixed $T>0$, define
\begin{align*}
\Omega_T=\{ \phi \in C([0,T];\mathcal{M}^{\alpha})\vert\ \phi(\xi,\cdot) \in C([0,T]) \ \mathrm{for \ each}\ \xi \in \mathbb{R}^3\}.
\end{align*}
Since $\mathcal{M}^{\alpha}$ is a complete metric space endowed with $\mathrm{dis}_{\alpha}(\phi,\psi)=\Vert \phi-\psi\Vert_{\mathcal{M}^{\alpha}}+\Vert \phi - \psi \Vert_{\alpha}$ (see \cite{MWY}), $\Omega_T$ is also a complete metric space endowed with
\begin{align*}
D_T(\phi,\psi)=\max_{t\in [0,T]} \mathrm{dis}_{\alpha}(\phi(t),\psi(t)).
\end{align*}
Let us think of the right-hand side of \eqref{FE cutoff} as the image $\Phi(\phi)(t)$ of an operator $\Phi$. By Lemma \ref{G}, $\Phi$ is well-defined on $\Omega_T$. We will prove that $\Phi$ is a contraction on $\Omega_T$ provided that $T$ is sufficiently small. Since $\mathcal{M}^{\alpha}$ is endowed with $\mathrm{dis}_{\alpha}$ and the contraction estimate for $\Vert \cdot \Vert_{\alpha}$ was obtained in \cite{C}, it is enough to consider the contraction estimate for $\Vert \cdot \Vert_{\mathcal{M}^{\alpha}}$.
\begin{lem}
If $\alpha_0 \le \alpha <p,\ T>0$, and $\phi \in \Omega_T$, then 
\begin{align*}
\Vert \Phi(\phi)(t)-1\Vert_{\mathcal{M}^{\alpha}} & \le C_{p,\alpha} (\delta_p t)^{\alpha /p}+2\max_{\tau \in [0,T]}\Vert \phi(\tau)-1 \Vert_{\mathcal{M}^{\alpha}},\\
\Vert \Phi(\phi)(t)-\Phi(\phi)(s) \Vert_{\mathcal{M}^{\alpha}} & \le C(\phi,T)\vert t-s \vert^{\alpha /p},
\end{align*}
where $s,t \in [0,T]$ and 
\begin{align*}
{\displaystyle C(\phi,T)=2C_{p,\alpha}\delta_p^{\alpha/p}(2+\gamma_2T)+3\gamma_{\alpha}T^{1-\alpha/p} \max_{t \in [0,T]} \Vert \phi(t)-1 \Vert_{\mathcal{M}^{\alpha}}}.
\end{align*}
\end{lem}

\begin{proof}
When $t\ge 0$, we have $\mathcal{G}(\phi)(0,t)=\gamma_2$, so 
\[
\Phi(\phi)(t)-1=I_1(t)+I_2(t)+I_3(t),
\]
where
\begin{align*}
I_1(t)&=e^{-\gamma_2t}\left[ e^{-\delta_p \vert \xi \vert^p t}\phi_0(\xi)-1\right],\\
I_2(t)&=\int^t_0 e^{-(\gamma_2+\delta_p\vert \xi \vert^p)(t-\tau)} \left[\mathcal{G}(\phi)(\xi,\tau)-\mathcal{G}(\phi)(0,\tau)\right] d\tau,\\
I_3(t)&=\gamma_2 \int^t_0 e^{-\gamma_2(t-\tau)} \left[ e^{-\delta_p\vert \xi \vert^p(t-\tau)}-1 \right] d\tau.
\end{align*}
Note that for any $a>0$, 
\begin{align*}
\int_{\mathbb{R}^3} \frac{1-e^{-a\vert \xi\vert^p}}{\vert\xi\vert^{3+\alpha}}d\xi=C_{p,\alpha}a^{\alpha/p}.
\end{align*}
Since
\begin{align*}
e^{-\delta_p \vert \xi \vert^p t}\phi_0(\xi)-1 = e^{-\delta_p\vert\xi\vert^pt}(\phi_0(\xi)-1)-(1-e^{-\delta_p\vert \xi \vert^pt})
\end{align*}
this gives
\begin{align*}
\int_{\mathbb{R}^3} \frac{\vert I_1(t)\vert}{\vert \xi \vert^{3+\alpha}} d\xi \le e^{-\gamma_2t} \left[ \Vert \phi_0-1\Vert_{\mathcal{M}^{\alpha}}+C_{p,\alpha}(\delta_pt)^{\alpha/p} \right].
\end{align*}
Next, by Lemma \ref{MWY-1},
\begin{align*}
\int_{\mathbb{R}^3} \frac{\vert I_2(t)\vert}{\vert \xi \vert^{3+\alpha}}d\xi \le & \gamma_{\alpha} \int^t_0 e^{-\gamma_2(t-\tau)}\Vert \phi(\tau)-1\Vert_{\mathcal{M}^{\alpha}} d\tau \\
\le & \gamma_{\alpha} \left( \frac{1-e^{-\gamma_2t}}{\gamma_2} \right) \max_{\tau \in [0,T]}\Vert \phi(\tau)-1\Vert_{\mathcal{M}^{\alpha}}.
\end{align*}
Finally, $I_3$ is estimated in a similar way as $I_1$, yielding
\begin{align*}
\int_{\mathbb{R}^3} \frac{\vert I_3(t)\vert}{\vert \xi \vert^{3+\alpha}} d\xi\le & \gamma_2\int^t_0 e^{-\gamma_2(t-\tau)}C_{p,\alpha}[\delta_p(t-\tau)]^{\alpha/p} d\tau \\
\le & C_{p,\alpha} (\delta_pt)^{\alpha/p} (1-e^{-\gamma_2t}).
\end{align*}
Combining these estimates we obtain the first inequality of the lemma.

To obtain the second inequality of the lemma, we estimate 
\begin{align*}
\Phi(\phi)(t)-\Phi(\phi)(s)=\sum_{i=1}^3 (I_i(t)-I_i(s))
\end{align*}
with $0\le s <t\le T$ for simplicity.
Since
\begin{align*}
I_1(t)-I_1(s)=-e^{-\gamma_2s}\left[ e^{-\delta_p\vert\xi\vert^ps}\left( \frac{1-e^{-\delta_p\vert\xi\vert^p(t-s)}}{\vert\xi\vert^{\alpha}}\right) \phi_0(\xi)\right.\\
\left. +(1-e^{-\gamma_2(t-s)})\left( \frac{e^{-\delta_p\vert\xi\vert^pt}\phi_0(\xi)-1}{\vert\xi\vert^{\alpha}} \right)\right],
\end{align*}
we have 
\begin{align*}
\int_{\mathbb{R}^3} \frac{\vert I_1(t)-I_1(s) \vert}{\vert \xi \vert^{3+\alpha}} d\xi \le & C_{p,\alpha} [\delta_p (t-s)]^{\alpha/p}\\
&+\gamma_2 (t-s) \left[ \Vert \phi_0-1\Vert_{\mathcal{M}^{\alpha}} +C_{p,\alpha}(\delta_p t)^{\alpha/p} \right]  \\
\le & C_{p,\alpha}\delta_p^{\alpha/p}(t-s)^{\alpha/p} \\
&+\gamma_2T^{1-\alpha/p}\Vert\phi_0-1\Vert_{\mathcal{M}^{\alpha}} +C_{p,\alpha} \gamma_2 \delta_p^{\alpha/p}T (t-s)^{\alpha/p}. 
\end{align*}
Next,
\begin{align*}
I_2(t)&-I_2(s)=\int^t_s e^{-(\gamma_2+\delta_p\vert\xi\vert^p)(t-\tau)}[\mathcal{G}(\phi)(\xi,\tau)-\mathcal{G}(\phi)(0,\tau)]d\tau \\
&-(1-e^{-\gamma_2(t-s)})\int^s_0 e^{-(\gamma_2+\delta_p\vert\xi\vert^p)(s-\tau)}[\mathcal{G}(\phi)(\xi,\tau)-\mathcal{G}(\phi)(0,\tau)]d\tau \\
&-(1-e^{-\delta_p\vert\xi\vert^p(t-s)})\int^s_0 e^{-\gamma_2(t-\tau)-\delta_p \vert\xi \vert^p (s-\tau)} [\mathcal{G}(\phi)(\xi,\tau)-\mathcal{G}(\phi)(0,\tau)]d\tau.
\end{align*}
Since  $b \in L^1(\mathbb{S}^2)$ we have $\vert\mathcal{G}(\phi)(\xi,\tau)-\mathcal{G}(\phi)(0,\tau)\vert \le 2\gamma_2$, which together with Lemma \ref{MWY-1} gives
\begin{align*}
\int_{\mathbb{R}^3} &\frac{\vert I_2(t)-I_2(s)\vert}{\vert \xi \vert^{3+\alpha}} d\xi \\
& \le  \gamma_{\alpha}\max_{\tau \in [0,T]} \Vert \phi(\tau)-1\Vert_{\mathcal{M}^{\alpha}} \left[ \int^t_s e^{-\gamma_2(t-\tau)}d\tau +\gamma_2(t-s)\int^s_0e^{-\gamma_2(s-\tau)}d\tau \right]\\
&\quad+C_{p,\alpha}[\delta_p(t-s)]^{\alpha/p}2\gamma_2 \int^s_0 e^{-\gamma_2(t-\tau)} d\tau\\
& \le (t-s)^{\alpha/p} \left[ 2\gamma_{\alpha}T^{1-\alpha/p}\max_{\tau \in [0,T]} \Vert \phi(\tau)-1\Vert_{\mathcal{M}^{\alpha}} +2C_{p,\alpha}\delta_p^{\alpha/p} \right].
\end{align*}
Finally,
\begin{align*}
I_3(t)-I_3(s)=&-\gamma_2\int^t_s e^{-\gamma_2(t-\tau)} (1-e^{-\delta_p\vert\xi\vert^p(t-\tau)})d\tau\\
&-\gamma_2(1-e^{-\delta_p\vert\xi\vert^p(t-s)}) \int^s_0 e^{-\gamma_2(t-\tau)-\delta_p\vert\xi\vert^p(s-\tau)}d\tau \\
&+\gamma_2(1-e^{-\gamma_2(t-s)})\int^s_0 e^{-\gamma_2(s-\tau)} (1-e^{-\delta_p\vert\xi\vert^p(s-\tau)})d\tau,
\end{align*}
so we have
\begin{align*}
\int_{\mathbb{R}^3} \frac{\vert I_3(t)-I_3(s) \vert}{\vert \xi \vert^{3+\alpha}} d\xi \le & \gamma_2 \int^t_s e^{-\gamma_2(t-\tau)}C_{p,\alpha}[\delta_p(t-\tau)]^{\alpha/p}d\tau \\
&+\gamma_2 C_{p,\alpha}[\delta_p (t-s)]^{\alpha/p} \int^s_0 e^{-\gamma_2(t-\tau)}d\tau \\
&+\gamma_2^2 (t-s)\int^s_0 e^{-\gamma_2(s-\tau)}C_{p,\alpha}[\delta_p (s-\tau)]^{\alpha/p} d\tau \\
\le & C_{p,\alpha}\delta_p^{\alpha/p}(t-s)^{\alpha/p} ( 1+\gamma_2T ).
\end{align*}
Together, these estimates yield the second inequality of the lemma.
\end{proof}

\begin{proof}[Proof of Theorem \ref{thm:Existence}]
To see that $\Phi$ is a contraction, note that
$\Phi$ is a mapping from $\Omega_T$ to $\Omega_T$ for each $T>0$ and 
\begin{equation}\label{eq:combining}
\max_{t \in [0,T]}\Vert \Phi(\phi)(t)-\Phi(\psi)(t)\Vert_{\mathcal{M}^{\alpha}} \le 2(1-e^{-\gamma_2T})\max_{t\in [0,T]} \Vert\phi(t)-\psi(t)\Vert_{\mathcal{M}^{\alpha}}.
\end{equation}
This follows in the same way as in Cho~\cite{C} (replace $\mathcal{K}^{\alpha}$ with $\mathcal{M}^{\alpha}$ and $\Vert \cdot \Vert_{\alpha}$ with $\Vert \cdot \Vert_{\mathcal{M}^{\alpha}}$, then this is obvious). Taking $0<T_0<\frac{\log 2}{\gamma_2}$ and combining \eqref{eq:combining} with the $\Vert \cdot \Vert_{\alpha}$-estimates in \cite{C}, we obtain
\begin{align*}
D_{T_0}(\Phi(\phi),\Phi(\psi)) <D_{T_0}(\phi,\psi).
\end{align*}
Therefore $\Phi$ is a contraction on $\Omega_{T_0}$. By using the Banach fixed point theorem, we get a solution of \eqref{FE cutoff}.
\end{proof}

\subsection{An a priori estimate and a stability estimate}

We observe that since $C([0,\infty);\mathcal{M}^{\alpha}) \subset C([0,\infty); \mathcal{K}^{\alpha})$, the a priori estimate deduced in \cite{C} also holds for a solution of \eqref{FE cutoff}:
\begin{lem}
Let $\alpha_0 \le \alpha <p,\  b \in L^1(\mathbb{S}^2)$, and $\phi_0 \in \mathcal{M}^{\alpha}$. If \eqref{FE cutoff} has a solution $\phi$ which satisfies $\phi \in C([0,\infty); \mathcal{M}^{\alpha})$ and $\phi(\xi,\cdot)\in C([0,\infty))$ for every $\xi \in \mathbb{R}^3$, then the following estimate holds for every $t \ge 0$
\begin{align}\label{maximum growth cutoff}
\sup_{\xi \in \mathbb{R}^3} e^{\delta_p\vert \xi \vert^pt}\vert \phi(\xi,t)\vert \le 1.
\end{align}
\end{lem}

Coupled with the existence of a local solution, estimate \eqref{maximum growth cutoff} gives the existence of a global solution of \eqref{FE cutoff}. We will now show a stability estimate for the uniqueness result.

\begin{thm}
Let $\gamma_2 <\infty$ and $\phi,\psi \in \mathcal{T}^\alpha (\mathbb{R}^3 \times [0,\infty))$ be solutions of \eqref{FE cutoff}, corresponding to initial data $\phi_0,\psi_0 \in \mathcal{M}^{\alpha}$ respectively. Then for all $t\ge 0$,
\begin{align*}
\left\Vert e^{\delta_p \vert \cdot \vert^pt}(\phi(\cdot,t)-\psi(\cdot,t)) \right\Vert_{\mathcal{M}^{\alpha}} \le e^{\lambda_{\alpha}t} \Vert \phi_0-\psi_0\Vert_{\mathcal{M}^{\alpha}}.
\end{align*}
\end{thm}

\begin{proof}

Define
\begin{align*}
U(\xi,t)=e^{(\gamma_2+\delta_p\vert\xi\vert^p)t} \left(\frac{\phi(\xi,t)-\psi(\xi,t)}{\vert \xi \vert^{3+\alpha}}\right)
\end{align*}
for $ \xi\neq 0$ and $U(0,t)=0$. We have
\begin{align}\label{gronwall}
U(\xi,t)=U(\xi,0)+\int^t_0 e^{(\gamma_2+\delta_p\vert\xi\vert^p)\tau} \left( \frac{\mathcal{G}(\phi)(\xi,\tau)-\mathcal{G}(\psi)(\xi,\tau)}{\vert\xi\vert^{3+\alpha}} \right)d\tau.
\end{align}
Since $\vert \xi\vert^p\le \vert \xi^+ \vert^p+\vert \xi^- \vert^p$ for $0\le p\le 2$, \eqref{maximum growth cutoff} gives
\begin{align*}
\int_{\mathbb{R}^3} & e^{\delta_p \vert\xi\vert^p\tau}  \frac{\vert \mathcal{G}(\phi)(\xi,\tau)-\mathcal{G}(\psi)(\xi,\tau)\vert}{\vert \xi \vert^{3+\alpha}}d\xi \\
& \le \int_{\mathbb{R}^3}\int_{\mathbb{S}^2} \frac{b\left( \frac{\xi\cdot \sigma}{\vert\xi\vert}  \right)e^{\delta_p(\vert\xi^+\vert^p+\vert\xi^-\vert^p)\tau}}{\vert\xi\vert^{3+\alpha}}  \left(\vert \phi^+ \vert\vert\phi^- -\psi^-\vert +\vert \psi^-\vert \vert \psi^+-\psi^+ \vert \right) d\sigma d\xi \\
& \le  \int_{\mathbb{R}^3} \int_{\mathbb{S}^2} \frac{b\left( \frac{\xi\cdot \sigma}{\vert\xi\vert}  \right)}{\vert\xi\vert^{3+\alpha}} \left( e^{\delta_p \vert \xi^+\vert^p\tau}\vert \phi^+-\psi^+ \vert +e^{\delta_p\vert\xi^-\vert^p\tau}\vert \phi^- -\psi^- \vert \right) d\sigma d\xi\\
&=J_1+J_2.
\end{align*}
Here we used the simplified notation $\phi(\xi^+,\tau)=\phi^+$, with $\phi^-$ and $\psi^\pm$ similarly defined. By using the change of variables used in the proof of \cite[Lemma 8.1]{C2} (see also \cite[Lemma 2.1]{MWY}), $J_1$ and $J_2$ are calculated as
\begin{align*}
J_1&=\int_{\mathbb{R}^3} \int_{\mathbb{S}^2}b\left( \frac{\xi}{\vert\xi\vert} \cdot \sigma \right) \frac{e^{\delta_p\vert\xi\vert^p\tau}}{\vert\xi\vert^{3+\alpha}} \vert \phi(\xi,\tau)-\psi(\xi,\tau) \vert \cos^\alpha \left( \frac{\theta}{2} \right) d\sigma d\xi,\\
J_2&=\int_{\mathbb{R}^3} \int_{\mathbb{S}^2}b\left( \frac{\xi}{\vert\xi\vert} \cdot \sigma \right) \frac{e^{\delta_p\vert\xi\vert^p\tau}}{\vert\xi\vert^{3+\alpha}} \vert \phi(\xi,\tau)-\psi(\xi,\tau) \vert \sin^\alpha \left( \frac{\theta}{2} \right) d\sigma d\xi,
\end{align*}
from which we conclude that
\begin{align*}
\left\Vert e^{\delta_p \vert \cdot \vert^p \tau}(\mathcal{G}(\phi)(\tau)-\mathcal{G}(\psi)(\tau)) \right\Vert_{\mathcal{M}^{\alpha}} \le \gamma_{\alpha} \left\Vert e^{\delta_p\vert\cdot\vert^p\tau}(\phi(\tau)-\psi(\tau))\right\Vert_{\mathcal{M}^{\alpha}}.
\end{align*}
For simplicity we denote $\left\Vert e^{\delta_p \vert \cdot \vert^pt}(\phi(\cdot,t)-\psi(\cdot,t))\right\Vert_{\mathcal{M}^{\alpha}}$ by $V(t)$. Integrating \eqref{gronwall} with respect to $x$, we obtain
\begin{align*}
e^{\gamma_2t}V(t) \le V(0) +\gamma_{\alpha} \int^t_0 e^{\gamma_2\tau} V(\tau) d\tau.
\end{align*}
The Gronwall inequality yields the result.
\end{proof}

\section{Global existence of a solution under the non-cutoff condition}\label{section:noncutoff}

In this section, we will construct a solution of \eqref{FE} without the cutoff assumption by using the results of the last section. For this purpose, first we prove two lemmas.

\begin{lem}\label{a priori}
Let $\alpha_0 \le \alpha <p\le 2,\ \mu_\alpha <\infty$, and $\phi_0 \in \mathcal{M}^\alpha$. If \eqref{FE} has a solution $\phi \in \mathcal{T}^\alpha (\mathbb{R}^3 \times [0,\infty))$ then $\phi$ satisfies
\begin{align*}
\Vert \phi(t)-1\Vert_{\mathcal{M}^{\alpha}} \le e^{C_0\mu_\alpha t} \left[ \Vert \phi_0 -1 \Vert_{\mathcal{M}^{\alpha}} + C_{p,\alpha}(\delta_p t)^{\alpha/p} \right],
\end{align*}
where $C_0$ is the same constant as in Lemma \ref{MWY-2}.
\end{lem}
\begin{proof}
It is easy to see that
\begin{align*}
\frac{\phi(\xi,t)-1}{\vert\xi\vert^{3+\alpha}}=\frac{e^{-\delta_p\vert\xi\vert^pt}\phi_0(\xi)-1}{\vert\xi\vert^{3+\alpha}}+\int^t_0e^{-\delta_p\vert\xi\vert^p(t-\tau)}\frac{\mathcal{B}(\phi)(\xi,\tau)}{\vert\xi\vert^{3+\alpha}}d\tau.
\end{align*}
By Lemma \ref{MWY-2} we obtain
\begin{align*}
\Vert \phi(t)-1\Vert_{\mathcal{M}^{\alpha}} \le \Vert\phi_0-1\Vert_{\mathcal{M}^{\alpha}} +C_{p,\alpha}(\delta_pt)^{\alpha/p} +C_0\mu_\alpha \int^t_0 \Vert \phi(\tau)-1\Vert_{\mathcal{M}^{\alpha}} d\tau. 
\end{align*}
The estimate now follows by applying the Gronwall inequality.
\end{proof}

For brevity, we will denote $C_0\mu_\alpha e^{C_0 \mu_\alpha T} \left[ \Vert \phi_0-1\Vert_{\mathcal{M}^{\alpha}} +C_{p,\alpha}(\delta_p T)^{\alpha/p}\right]$ by $\tilde{C_0}$.
\begin{lem}\label{time-continuity}
Fix $T>0$. Under the assumptions of Lemma \ref{a priori} we have for all $s,t \in [0,T]$
\begin{align*}
\Vert \phi(t)-\phi(s)\Vert_{\mathcal{M}^{\alpha}} \le C'(\phi_0,T)\vert t-s \vert^{\alpha/p},
\end{align*}
where
\begin{align*}
C'(\phi_0,T)=C_{p,\alpha}\delta_p^{\frac{\alpha}{p}}+\left(1+\frac{1}{1-\alpha/p}\right)\tilde{C_0}T^{1-\alpha/p}.
\end{align*}
\end{lem}

\begin{proof}
Set $ 0\le s <t \le T$ and define $K_1, K_2, K_3$ by
\begin{align*}
\phi(\xi,t)-\phi(\xi,s)&=\left( e^{-\delta_p\vert\xi\vert^pt} -e^{-\delta_p\vert\xi\vert^ps} \right)\phi_0(\xi)+\int^t_s e^{-\delta_p\vert\xi\vert^p(t-\tau)}\mathcal{B}(\phi)(\xi,\tau) d\tau \\
&\quad +\int^s_0 \left( e^{-\delta_p\vert\xi\vert^p(t-\tau)} -e^{-\delta_p\vert\xi\vert^p(s-\tau)} \right) \mathcal{B}(\phi)(\xi,\tau)d\tau \\
&=K_1+K_2+K_3.
\end{align*}
Obviously
\begin{align*}
\int_{\mathbb{R}^3} \frac{\vert K_1 \vert}{\vert\xi\vert^{3+\alpha}} d\xi \le \int_{\mathbb{R}^3} \frac{e^{-\delta_p\vert\xi\vert^ps}\left( 1- e^{-\delta_p\vert\xi\vert^p(t-s)}\right)}{\vert\xi\vert^{3+\alpha}}d\xi \le C_{p,\alpha} [\delta_p(t-s)]^{\alpha/p}. 
\end{align*}
By Lemma \ref{MWY-2} and Lemma \ref{a priori},
\begin{align*}
\int_{\mathbb{R}^3} \frac{\vert K_2 \vert}{\vert\xi\vert^{3+\alpha}}d\xi \le  C_0\mu_\alpha \int^t_s \Vert\phi(\tau)-1\Vert_{\mathcal{M}^{\alpha}} d\tau 
\le  \tilde{C_0}(t-s).
\end{align*}
For a fixed $\tau \in [0,s)$, $f(r)=e^{-(s-\tau)r}-e^{-(t-\tau)r} \ (r\ge 0)$ is a nonnegative function, satisfying 
\begin{align*}
\max_{r\ge 0} f(r)=\frac{t-s}{t-\tau} \left( \frac{s-\tau}{t-\tau} \right)^{\frac{s-\tau}{t-s}}.
\end{align*}
Furthermore, we have
\begin{equation}
\left( \frac{s-\tau}{t-\tau} \right)^{\frac{s-\tau}{t-s}} \le  \left( \frac{t-\tau}{t-s} \right)^{1-\alpha/p}
\end{equation} 
because the left-hand side is smaller than $1$ and the right-hand side is larger than $1$.
Thus
\begin{align*}
\int_{\mathbb{R}^3} \frac{\vert K_3\vert}{\vert\xi\vert^{3+\alpha}}d\xi \le &  \int_{\mathbb{R}^3} \int^s_0 \left(\frac{t-s}{t-\tau}\right)^{\alpha/p}\cdot \frac{\vert \mathcal{B}(\phi)(\xi,\tau) \vert}{\vert\xi\vert^{3+\alpha}}d\xi d\tau \\
\le & (t-s)^{\alpha/p}\cdot  \tilde{C_0} \left[ -\frac{1}{1-\alpha/p}(t-\tau)^{1-\alpha/p}\right]^s_{\tau=0}\\
\le & (t-s)^{\alpha/p}\tilde{C_0} \frac{t^{1-\alpha/p}}{1-\alpha/p}
\le  \frac{T^{1-\alpha/p}}{1-\alpha/p}\tilde{C_0}(t-s)^{\alpha/p} .
\end{align*}
Combining these estimates, the result follows.
\end{proof}

\begin{proof}[Proof of Theorems \ref{existence}-\ref{thm:original}]
We follow \cite{MWY}. For each $n\in \mathbb{N}$, define $b_n(\cos \theta)=\min\{b(\cos \theta), n\}$ and $\phi_n$ as a solution of \eqref{FE cutoff} with $b$ replaced by $b_n$. Define $\lambda_{\alpha,n}$ in the same manner. Then we have
\begin{align*}
\Vert \phi_n(t)-1\Vert_{\mathcal{M}^{\alpha}} \le e^{\lambda_{\alpha,n}t} \Vert\phi_0-1\Vert_{\mathcal{M}^{\alpha}} \le e^{\lambda_\alpha t} \Vert \phi_0-1\Vert_{\mathcal{M}^{\alpha}}.
\end{align*}
In \cite{C}, the Ascoli-Arzel\`{a} theorem was used to extract a subsequence of the series of functions $\{ \phi_n\}$, which converges to a solution $\phi$ of \eqref{FE} in $ C([0,\infty);\mathcal{K}^{\alpha})$. Since $\phi$ is a unique solution in $ \mathcal{S}_p^\alpha(\mathbb{R}^3 \times [0,\infty))$, we can extract such a subsequence. We denote this subsequence also by $\{ \phi_n\}$. For all $\delta >0$, we have
\begin{align*}
\int_{\delta< \vert\xi\vert <\delta^{-1}} \frac{\vert\phi(\xi,t)-1\vert}{\vert\xi\vert^{3+\alpha}}d\xi = \lim_{n\rightarrow \infty} \int_{\delta< \vert\xi\vert <\delta^{-1}} \frac{\vert\phi_n(\xi,t)-1\vert}{\vert\xi\vert^{3+\alpha}}d\xi \le e^{\lambda_\alpha t} \Vert \phi_0-1\Vert_{\mathcal{M}^{\alpha}}.
\end{align*}
The Lebesgue convergence theorem therefore implies that $\phi(\cdot,t) \in \mathcal{M}^{\alpha}$. That $\phi \in C([0,\infty);\mathcal{M}^{\alpha})$ follows from Lemma \ref{time-continuity}. Finally, by differentiating \eqref{FE} with respect to $t$, we obtain $\phi \in \mathcal{T}^{\alpha}(\mathbb{R}^3\times [0,\infty))$. For the proof of the maximum growth estimate and the stability estimate of a solution, we take two sequences $\{\phi_n\},\{\psi_n\} \subset \mathcal{S}^{\alpha}(\mathbb{R}^3 \times [0,\infty))$ which converge to solutions $\phi, \psi \in \mathcal{S}^\alpha(\mathbb{R}^3 \times [0,\infty))$ with initial data $\phi_0,\psi_0\in \mathcal{M}^{\alpha}$, respectively. Then for any $t\ge 0$
\begin{equation}
\sup_{\xi \in \mathbb{R}^3} e^{\delta_p\vert \xi \vert^pt}\vert \phi_n(\xi,t)\vert  \le 1\  \text{and }
\left\Vert e^{\delta_p \vert \cdot \vert^pt}(\phi_n(\cdot,t)-\psi_n(\cdot,t)) \right\Vert_{\mathcal{M}^{\alpha}}  \le e^{\lambda_{\alpha}t} \Vert \phi_0-\psi_0\Vert_{\mathcal{M}^{\alpha}}.
\end{equation}
Letting $n\rightarrow \infty$ in these inequalities yields Theorem \ref{stability}.
Once \eqref{maximal growth} is proved, we immediately conclude that $\phi(\xi, t)$ in Theorem \ref{existence}  belongs to $L^1$ and to any polynomially weighted $L^2$ space with respect to $\xi$. The Fourier inversion formula implies that $\phi(\xi,t)$ has a unique probability density $f(v,t)$ vanishing at infinity, and the Sobolev embedding theorem shows that $f(v,t)$ is a smooth function of $v$ for every positive $t$. 
The proof is completed by repeating the arguments of \cite[Section 3.3]{MWY},
the details of which are omitted.
\end{proof}

\section{Non-existence of a solution in the case $p\le\alpha \le2$}
In \cite{C},  a solution of \eqref{FE} was found in $\mathcal S^\alpha$ for initial data in $\mathcal{K}^{\alpha}$ when $0<\alpha\le p\le 2$ and non-existence of solutions was proved when $p<\alpha\le2$. 
However, in our case the existence of a solution in $\mathcal T^\alpha$
for initial data in $\mathcal{M}^{\alpha}$ fails also when $\alpha=p$.
Here we recall Lemma \ref{stable process}. Since $\mathcal{M}^{\alpha} =\mathcal{F}(P_\alpha(\mathbb{R}^3))$ when $0 <\alpha <2$ and $\alpha \neq 1$, the following theorem precisely reflects the moment estimates of $f_p(v,t)$.

\begin{thm}
Let $p\le\alpha \le2$. Then for any $T>0$ and $\phi_0 \in \mathcal{M}^{\alpha}$, there is no solution of \eqref{FE} which belongs to $\mathcal{T}^{\alpha}(\mathbb{R}^3\times [0,T])$.
\end{thm}
\begin{proof}
Assume to reach a contradiction that $\phi\in \mathcal{T}^{\alpha}(\mathbb{R}^3\times [0,T])$ is a solution of \eqref{FE}. It is easy to check that
\begin{align*}
1-e^{-\delta_p\vert\xi\vert^pt}=e^{-\delta_p\vert\xi\vert^pt}(\phi_0(t)-1)+1-\phi(\xi,t)+\int^t_0 e^{-\delta_p\vert\xi\vert^p(t-\tau)}\mathcal{B}(\phi)(\xi,\tau)d\tau.
\end{align*}
Clearly,
\begin{align*}
\int_{\mathbb{R}^3} \frac{1-e^{-\delta_p\vert\xi\vert^pt}}{\vert\xi\vert^{3+\alpha}}d\xi \le  2\max_{t\in [0,T]}\Vert \phi(t)-1\Vert_{\mathcal{M}^{\alpha}} +\int^t_0 C_0\mu_\alpha \Vert \phi(\tau)-1 \Vert_{\mathcal{M}^{\alpha}} d\tau <\infty,
\end{align*}
which contradicts the fact that $\int_{\mathbb{R}^3} \frac{1-e^{-\delta_p\vert\xi\vert^pt}}{\vert\xi\vert^{3+\alpha}}d\xi$ diverges.
\end{proof}

\section*{Acknowledgements}
The author wishes to express his profound gratitude to
Yoshinori Morimoto for many fruitful discussions and suggestions. The author also wishes to thank Yong-Kum Cho for valuable comments which helped improve earlier versions of the results and proofs.

\end{document}